\documentclass[11pt]{article}
\usepackage{amssymb,amsmath,amsthm,amsfonts}
\usepackage{graphicx}
\usepackage{epsfig}

\usepackage{amssymb,amsmath,amsthm,amsfonts,mathrsfs, bbm,dsfont}

\theoremstyle{plain}

\theoremstyle{definition}
\newtheorem{definition}{Definition}[section]

\newtheorem{remark}{Remark}[section]

\newtheorem{theorem}{Theorem}[section]

\newtheorem{proposition}{Proposition}[section]

\numberwithin{equation}{section}

\begin{document}
\openup 0.8\jot
\title{\Large\bf $C^*$-basic construction on field algebras of $G$-spin models \thanks{Tis work is supported by the National Natural Science
Foundation of China (Grants Nos. 11701423, 61701343 and
61771294)} }

\author{Xin Qiaoling$^1$, Jiang Lining$^2$, Cao Tianqing$^3$}
\date{}
\maketitle\begin{center}
\begin{minipage}{16cm}
{\small \it
$1$
School of Mathematical Sciences, Tianjin Normal University, Tianjin 300387, China

$2$
School of Mathematics and Statistics, Beijing Institute
of Technology, Beijing 100081, China

$3$
School of Mathematical Sciences, Tiangong University, Tianjin 300387, China
}
\end{minipage}
\end{center}
\vspace{0.05cm}
\begin{center}
\begin{minipage}{16cm}
{\small {\bf Abstract}: Let $G$ be a finite group. Starting from the field algebra ${\mathcal{F}}$ of $G$-spin models, one can construct the crossed product $C^*$-algebra ${\mathcal{F}}\rtimes D(G)$
such that it coincides with the $C^*$-basic construction for the field algebra ${\mathcal{F}}$ and the $D(G)$-invariant subalgebra of ${\mathcal{F}}$, where $D(G)$ is the quantum double of $G$. Under the natural $\widehat{D(G)}$-module action on ${\mathcal{F}}\rtimes D(G)$,
the iterated crossed product $C^*$-algebra
can be obtained, which is $C^*$-isomorphic to the $C^*$-basic construction for ${\mathcal{F}}\rtimes D(G)$ and the field algebra ${\mathcal{F}}$. Furthermore, one can show that the iterated crossed product $C^*$-algebra is a new field algebra
and give the concrete structure with the order and disorder operators.
}
\endabstract
\end{minipage}\vspace{0.10cm}
\begin{minipage}{16cm}
\end{minipage}
\end{center}
\begin{center} \vspace{0.01cm}
\end{center}







\section{Introduction}
In \cite{Jones}, V.F.R.Jones introduced the notion of indices for inclusions of II$_1$ factors,
which opened a completely new aspect of operator algebras involving various other
fields of mathematics and physics, including topology, quantum physics, dynamical systems, noncommutative
geometry, etc.
He established a method to
study the inner structures or the outer structures of given II$_1$-factors, preserving
certain quantity, called the Jones index. And he enlarged the pure algebraic Galois
theory to operator algebra, and this has a significant impact on mathematics and applied
mathematics. The interesting fact is that not only his theory but also the techniques
he used in his theory are important mathematically and applicable to other fields \cite{V.F.R.J}.
Subsequently, there are lots of attempts to extend the original Jones theory
by various mathematicians. For instance, M.Pimsner and S.Popa in \cite{M.PP} introduced the notion of the probabilistic index for
a conditional expectation, which is the best constant of the so-called Pimsner-Popa
inequality. In \cite{H.Ko}, H.Kosaki discussed another way to define index for a normal semifinite faithful conditional expectation of an arbitrary factor
onto a subfactor exploiting spatial theory of Connes and the theory of operatorvalued weights.

Though the probabilistic index works perfectly for analytic purposes even in the case of $C^*$-algebras (see, e.g. \cite{Popa}), it is not always suitable for algebraic operations
such as the basic construction in the $C^*$-case. Inspired by the Pimsner-Popa basis in the sense of \cite{M.PP}, Kosaki's index formula in the sense of \cite{H.Ko}, and the Casimir elements for
semi-simple Lie algebras, Y.Watatani \cite{Y.Wat} proposed to assume existence of a quasi-basis for a conditional expectation, a generalization of
the Pimsner-Popa basis in the von Neumann algebra case, to analyze inclusions of $C^*$-algebras. With a quasi-basis, Watatani successfully introduced a $C^*$-version of
basic construction, which is closely related to $K$-theory of $C^*$-algebras \cite{M.Izu,Y.Wat}.
Roughly speaking, the (original or extended) Jones index
theory is the study of certain elements measuring the maximal number of disjoint copies of topological $*$-subalgebra in a given topological $*$-algebra.

On the other hand, one-dimensional $G$-spin models, as a testing ground for applications of quantum group symmetries,
have an order-disorder type of
quantum symmetry given by the quantum double $D(G)$ of a finite group $G$,
defined by Drinfel'd in the context of finding solutions to the quantum
Yang-Baxter equation arising from statistical mechanics \cite{V.G.Dr}, which generalizes the
$Z(2)\times Z(2)$ symmetry of the lattice Ising model.
In \cite{K.Szl}, the implementation of the Doplicher-Haag-Roberts theory of superselection sectors \cite{DHR1,DHR2,DHR3,DHR4} to $G$-spin models has been carried out. In this approach the symmetries can be reflected as the $D(G)$-invariant subalgebra ${\mathcal{A}}$ of the field algebra ${\mathcal{F}}$ in $G$-spin models, called the observable algebra.

The paper studies the Jones basic construction on field algebras of $G$-spin models, and is organized as follows.

In Section 2, we collect the necessary definitions and facts about $G$-spin models, such as the quantum double $D(G)$, the field algebra ${\mathcal{F}}$ and the Hopf action of the symmetry algebra $D(G)$ on ${\mathcal{F}}$, and then we also give a brief description of the $C^*$-basic construction for $C^*$-algebras.

In Section 3, under the conditional expectation $E$ from the field algebra ${\mathcal{F}}$ onto the $D(G)$-invariant subalgebra ${\mathcal{A}}$, we can construct the crossed product $C^*$-algebra
${\mathcal{F}}\rtimes D(G)$, and then prove that this algebra is $C^*$-isomorphic to the $C^*$-algebra $\langle {\mathcal{F}},e_{\mathcal{A}}\rangle_{C^*}$
constructed from the $C^*$-basic construction for
the inclusion ${\mathcal{A}}\subseteq{\mathcal{F}}$.

In Section 4, we construct the iterated crossed product $C^*$-algebra ${\mathcal{F}}\rtimes D(G)\rtimes\widehat{D(G)}$, which is canonically isomorphic to $M_{|G|^2}({\mathcal{F}})$ by Takai duality \cite{H.Tak}, and we prove that the $C^*$-algebra ${\mathcal{F}}\rtimes D(G)\rtimes\widehat{D(G)}$ is $C^*$-isomorphic to
$\langle{\mathcal{F}}\rtimes D(G),e_2\rangle_{C^*}$ constructed from the $C^*$-basic construction
for the inclusion ${\mathcal{F}}\subseteq{\mathcal{F}}\rtimes D(G)$. Moreover,
we give the concrete description of the new field algebra ${\mathcal{F}}\rtimes D(G)\rtimes\widehat{D(G)}$ by means of the order and disorder operators.

All the algebras in this paper will be unital associative algebras over the complex field ${\Bbb C}$. The unadorned tensor product $\otimes$ will stand for the usual tensor product over ${\Bbb C}$.
 For general results on Hopf algebras please refer to the books of Abe \cite{E.Abe} and Sweedler \cite{M.E.Sw}. We shall follow their notations, such as $S$, $\bigtriangleup$, $\varepsilon$ for the antipode, the comultiplication and the counit, respectively. Also we shall use the so-called
``Sweedler-type notation" for the image of $\bigtriangleup$. That is
\begin{eqnarray*}
    \begin{array}{c}
    \bigtriangleup(a)=\sum\limits_{(a)}a_{(1)}\otimes a_{(2)}.
     \end{array}
\end{eqnarray*}

\section{$G$-spin models and the $C^*$-basic construction}
We first recall the main features of $G$-spin models, considered in
the $C^*$-algebraic framework for quantum lattice systems, and then give the $C^*$-basic construction for $C^*$-algebras in $G$-spin models.

\subsection{Definitions and preliminary results}
Assume that $G$ is a finite group with a unit $u$. The $G$-valued spin configuration on the two-dimensional square lattices is the map $\sigma\colon{\Bbb Z}^2\rightarrow G$ with Euclidean action functional: $$S(\sigma)=\sum \limits_{(x,y)} f(\sigma_x^{-1}\sigma_y),$$ in which the summation runs over the nearest neighbor pairs in ${\Bbb Z}^2$ and $f\colon G\rightarrow {\Bbb R}$ is a function of the positive type. This kind of classical statistical systems are called $G$-spin models \cite{S.Dop,Qiao}. And such models provide the simplest examples of lattice field theories exhibiting quantum symmetry.
In general, $G$-spin models with an Abelian group $G$ are known to have a symmetry group $G\times \widetilde{G}$, where $\widetilde{G}$ is the group of characters of $G$. If $G$ is non-Abelian, the models have a symmetry of a quantum double
 $D(G)$ \cite{K.A.Da}, which is defined as follows.

\begin{definition}
Let $C(G)$ be the algebra of complex valued functions on
$G$ and consider the adjoint action of $G$ on $C(G)$ according to $\alpha_g\colon f\mapsto f\circ Ad(g^{-1})$.
The quantum double $D(G)$ is defined as the crossed product $D(G) = C(G)\rtimes_{\alpha} G$ of
$C(G)$ by this action. In terms of generators $D(G)$ is the algebra generated by elements
$U_g$ and $V_h$ $(g,h\in G)$, with the relations
\begin{eqnarray*}
    \begin{array}{rcl}
    U_g U_h&=&\delta_{g,h}U_g,\\[5pt]
    V_g V_h&=&V_{gh},\\[5pt]
    V_h U_g&=&U_{hgh^{-1}}V_h,
     \end{array}
\end{eqnarray*}
and the identification $\sum\limits_{g\in G}U_g=V_u=1$, where
$\delta_{g,h}=\left\{\begin{array}{cc}
                          1 ,& {\mathrm{if}} \   g=h, \\[5pt]
                          0 ,& {\mathrm{if}} \  g\neq h.
                        \end{array}\right.$
\end{definition}

It is easy to see that $D(G)$ is of finite dimension, where as a convenient basis one
may choose $U_gV_h, g,h\in G$, multiplying according to
$U_{g_1}V_{h_1}U_{g_2}V_{h_2}=
\delta_{g_1h_1,h_1g_2}U_{g_1}V_{h_1h_2}.$

Here and from now on, by $(g,h)$ we always denote the element $U_gV_h$ for notational convenience.

Also, the structure maps are given by
 \begin{eqnarray*}
    \begin{array}{rcll}
\bigtriangleup(g,h)&=&\sum \limits_{t\in G}(t,h)\otimes (t^{-1}g,h),& (\mathrm{coproduct})\\[5pt]
\varepsilon(g,h)&=&\delta_{g,u},& (\mathrm{counit})\\[5pt]
S(g,h)&=&(h^{-1}g^{-1}h,h^{-1}),& (\mathrm{antipode})
     \end{array}
\end{eqnarray*}
on the linear basis $\{(g,h),g,h\in G\}$ and are extended in $D(G)$ by linearity.
One can prove that $D(G)$ is a Hopf algebra, with a unique element $E=\frac{1}{|G|}\sum\limits_{g\in G}(u,g)$, called an integral element, satisfying for any $a\in D(G)$,
$$aE=Ea=\varepsilon(a)E.$$
Moreover, with the definition
 \begin{eqnarray*}
    \begin{array}{c}
    (g,h)^*=(h^{-1}gh,h^{-1}),
         \end{array}
\end{eqnarray*}
and the appropriate extension, $D(G)$ is a semisimple $*$-algebra of finite dimension \cite{M.E.Sw}, which implies that $D(G)$ becomes a Hopf $C^*$-algebra.

As in the traditional case, one can define the local quantum field algebra as follows.

\begin{definition}
 The local field algebra of a $G$-spin model ${\mathcal{F}}_{\mathrm{loc}}$ is an associative algebra with a unit $I$ generated by
 $\{\delta_g(x), \rho_h(l)\colon g, h\in G, x\in {\Bbb Z}, l\in {\Bbb Z}+\frac{1}{2}\}$ subject to
\begin{eqnarray*}
    \begin{array}{rcl}
\sum\limits_{g\in G}\delta_g(x)&=&I=\rho_u(l),\\[5pt]
\delta_{g_1}(x)\delta_{g_2}(x)&=&\delta_{g_1,g_2}\delta_{g_1}(x),\\[5pt]
 \rho_{h_1}(l)\rho_{h_2}(l)&=&\rho_{h_1h_2}(l),\\[5pt]
\delta_{g_1}(x)\delta_{g_2}(x')&=&\delta_{g_2}(x')\delta_{g_1}(x),\\[5pt]
 \rho_h(l)\delta_g(x)&=& \left\{\begin{array}{cc}
                          \delta_{hg}(x)\rho_h(l), &\mbox{if}\ l<x, \\[5pt]
                          \delta_g(x)\rho_h(l), & \mbox{if}\ l>x,
                        \end{array}\right.\\[5pt]
 \rho_{h_1}(l)\rho_{h_2}(l')&=& \left\{\begin{array}{cc}
                          \rho_{h_2}(l')\rho_{{h_2}^{-1}h_1{h_2}}(l), & \mbox{if}\ l>l', \\[5pt]
                          \rho_{{h_1}{h_2}{h_1}^{-1}}(l')\rho_{h_1}(l), & \mbox{if}\ l<l',
                        \end{array}\right.

    \end{array}
\end{eqnarray*}
for $x, x' \in {\Bbb Z},$ $l, l'\in {\Bbb Z}+\frac{1}{2}$ and $h_1, h_2, g_1,g_2\in G$.
\end{definition}

The $*$-operation is defined on the generators as $\delta_g^\ast(x)=\delta_g(x), \ \rho_h^\ast(l)=\rho_{h^{-1}}(l)$ and can be extended to an involution on ${\mathcal{F}}_{\mathrm{loc}}$. In this way, ${\mathcal{F}}_{\mathrm{loc}}$ becomes a unital $*$-algebra.
Using the $C^*$-inductive limit \cite{B.R.Li,K.Szl}, ${\mathcal{F}}_{\mathrm{loc}}$ can be extended to
a $C^*$-algebra ${\mathcal{F}}$, called the field algebra of $G$-spin models.

There is an action $\gamma$ of $D(G)$ on ${\mathcal{F}}$ in the following. For $x\in {\Bbb Z}$, $l\in {\Bbb Z}+\frac{1}{2}$ and $g,h\in G$, set
\begin{eqnarray*}
    \begin{array}{rcll}
(g,h)\delta_f(x)&=&\delta_{g,u}\delta_{hf}(x), & \forall f\in G,\\
(g,h)\rho_t(l)&=&\delta_{g,hth^{-1}}\rho_g(l), & \forall t\in G.
 \end{array}
\end{eqnarray*}
The map $\gamma$ can be extended for products of generators inductively in the number of generators by the rule
\begin{eqnarray*}
    \begin{array}{c}
    (g,h)(fT)=\sum\limits_{(g,h)}(g,h)_{(1)}(f)(g,h)_{(2)}(T),
 \end{array}
\end{eqnarray*}
where $f$ is one of the generators in ${\mathcal{F}}_{\mathrm{loc}}$ and $T$ is a finite product of generators. Finally, it is linearly extended both in $D(G)$ and ${\mathcal{F}}_{\mathrm{loc}}$.

\begin{proposition}$^{\cite{K.Szl}}$
The field algebra ${\mathcal{F}}$ is a $D(G)$-module algebra with respect to the map $\gamma$. Namely, the map $\gamma$ satisfies the following relations:
\begin{eqnarray*}
    \begin{array}{rcl}
(ab)(T)&=&a(b(T)),\\[5pt]
a(T_1T_2)&=&\sum\limits_{(a)}a_{(1)}(T_1)a_{(2)}(T_2),\\[5pt]
a(T^*)&=&(S(a^*)(T))^*,
 \end{array}
\end{eqnarray*}
for $ a, b\in D(G), \ T_1, T_2, T\in {\mathcal{F}}$.
\end{proposition}

Set $${\mathcal{A}}\ =\ \{F\in {\mathcal{F}}\colon a(F)=\varepsilon(a)(F), \ \forall a\in D(G)\}.$$ We call it an observable algebra in the field algebra ${\mathcal{F}}$ of $G$-spin models. Furthermore, one can show that ${\mathcal{A}}$ is a nonzero $C^*$-subalgebra of ${\mathcal{F}}$, and $${\mathcal{A}}\ =\ \{F\in {\mathcal{F}}\colon E(F)=F\}\equiv E({\mathcal{F}}).$$
Indeed, from the following proposition, one can see that ${\mathcal{A}}$ is a $C^*$-subalgebra of ${\mathcal{F}}$.

\begin{proposition}$^{\cite{K.Szl}}$
 The map $E\colon{\mathcal{F}}\rightarrow{\mathcal{A}}$ satisfies the following conditions:

 $(1)$\ $E(I)=I$ where $I$ is the unit of ${\mathcal{F}}$;

 $(2)$\ (bimodular property) $\forall \ F_1, F_2\in{\mathcal{A}}, \ F\in {\mathcal{F}}$,
 $$E(F_1FF_2)=F_1E(F)F_2;$$

 $(3)$\ $E$ is positive.
\end{proposition}

In the following a linear map $\Gamma$ from a unital $C^*$-algebra $B$ onto its unital $C^*$-subalgebra $A$ with properties (1)-(3) in Proposition 2.1 is called a conditional expectation. If $\Gamma$ is a conditional expectation from $B$ onto $A$, then $\Gamma$ is a projection of norm one \cite{O.Bra}. In addition, if $E(Bb)=0$ implies $b=0$, for $b\in B$, then we say $E$ is faithful.

We next review some basic facts about the index for $C^*$-algebras in \cite{Y.Wat}.

\begin{definition}
Let $\Gamma$ be a conditional expectation from a unital $C^*$-algebra $B$ onto its unital $C^*$-subalgebra $A$. A finite family $\{(u_1,v_1), (u_2,v_2),\cdots, (u_n,v_n)\}\subseteq B\times B$ is called a quasi-basis for $\Gamma$ if for all $b\in B$,
 \begin{eqnarray*}
    \begin{array}{c}
\sum\limits_{i=1}^{n}u_i\Gamma(v_ib)=b=\sum\limits_{i=1}^{n}\Gamma(b u_i)v_i.
    \end{array}
\end{eqnarray*}
Furthermore, if there exists a quasi-basis for $\Gamma$, we call $\Gamma$ of index-finite type. In this case we define the index of $\Gamma$ by
 \begin{eqnarray*}
    \begin{array}{c}
{\mathrm{Index}}\ \Gamma=\sum\limits_{i=1}^nu_iv_i.
      \end{array}
\end{eqnarray*}
\end{definition}

\begin{remark}
(1) If $\Gamma$ is a conditional expectation of index-finite type, then the $C^*$-index ${\mathrm{Index}}\ \Gamma$ is a central element of $B$ and does not depend on the choice of quasi-basis.
In particular, if $A\subseteq B$ are simple unital $C^*$-algebras, then we can choose one of the form $\{(w_i,w_i^*)\colon i=1,2,\cdots, n\}$, which shows that ${\mathrm{Index}}\ \Gamma$ is a positive element \cite{Y.Wat}.

(2) Let $N \subseteq M$ be factors of type II$_1$ and $\Gamma\colon M\rightarrow N$ the canonical conditional expectation determined by the unique normalized trace on $M$, then ${\mathrm{Index}}\ \Gamma$ is exactly Jones index $[M:N]$ based on the coupling constant \cite{M.PP}. More generally, let $M$ be a ($\sigma$-finite) factor with a
subfactor $N$ and $\Gamma$ a normal conditional expectation from $M$
onto $N$, then $\Gamma$ is of index-finite if and only if ${\mathrm{Index}}\ \Gamma$ is finite in the sense of Ref. \cite{H.Ko}, and the values of ${\mathrm{Index}}\ \Gamma$ are equal.
\end{remark}

\subsection{The $C^*$-basic construction for the inclusion ${\mathcal{A}}\subseteq{\mathcal{F}}$}
This section will give a concrete description of the $C^*$-basic construction for the inclusion ${\mathcal{A}}\subseteq {\mathcal{F}}$ and some properties about the Jones projection.

Let $\Gamma\colon B\rightarrow A$ be a faithful conditional expectation. Then $B_A$(viewing $B$ as a right $A$-module) is a pre-Hilbert module over $A$ with an $A$-valued inner product
$\langle x,y\rangle=\Gamma(x^*y)$ for $x,y\in  B_A.$
Let $\overline{B_A}$ be the completion of $B_A$ with respect to the norm on $B_A$ defined by
$$\|x\|_{B_A}=\|\Gamma(x^*x)\|_A^{\frac{1}{2}}, \ x\in B_A.$$
Then $\overline{B_A}$ is a Hilbert $C^*$-module over $A$. Since $\Gamma$ is faithful, the canonical map
$B\rightarrow {\overline{B_A}}$ is injective. Let $L_A(\overline{B_A})$ be the set of all (right) $A$-module homomorphisms
$T\colon \overline{B_A}\rightarrow \overline{B_A}$ with an adjoint $A$-module homomorphism $T^*\colon \overline{B_A}\rightarrow \overline{B_A}$ such that
$$\langle T\xi,\eta\rangle=\langle\xi, T^*\eta\rangle.$$
Then $L_A(\overline{B_A})$ is a $C^*$-algebra with the operator norm
$$\|T\|\ =\ \sup\{\|T\xi\|\colon \|\xi\|=1\}.$$
There is an injective $*$-homomorphism $\lambda\colon B\rightarrow L_A(\overline{B_A})$ defined by
$\lambda(b)x=bx$
for $x\in B_A$ and $b\in B$, so that $B$ can be viewed as a $C^*$-subalgebra of $L_A(\overline{B_A})$.
Note
that the map $\gamma_A \colon B_A \rightarrow B_A$ defined by
$\gamma_A(x)=\Gamma(x)$ for $x\in B_A$
is bounded and thus it can be extended to a bounded linear operator on $\overline{B_A}$, denoted by
$\gamma_A$ again. Then $\gamma_A\in L_A(\overline{B_A})$ and $\gamma_A =\gamma_A^2
=\gamma_A^*$; that is, $\gamma_A$ is a projection in
$L_A(\overline{B_A})$. From now on we call $\gamma_A$ the Jones projection of $\Gamma$.
The (reduced) $C^*$-basic construction is a $C^*$-subalgebra of $L_A(\overline{B_A})$ defined to be
$$ \langle B,\gamma_A\rangle_{C^*}= \overline{\mbox{span}\{\lambda(x)\gamma_A\lambda(y)\in L_A(\overline{B_A}) \colon x, y \in B \}}^{\|\cdot\|}.$$

For the conditional expectation $E\colon {\mathcal{F}}\rightarrow{\mathcal{A}}$, we shall consider the $C^*$-basic construction
$\langle{\mathcal{F}},e_{\mathcal{A}}\rangle_{C^*}$, which is a $C^*$-subalgebra of
$L_{\mathcal{A}}(\overline{{\mathcal{F}}})$
linearly generated by
$\{\lambda(x)e_{\mathcal{A}}\lambda(y)\colon x,y\in {\mathcal{F}}\}$,
where $\overline{\mathcal{F}}$ is the completion of ${\mathcal{F}}_{\mathcal{A}}$ with respect to the norm $\|x\|_{{\mathcal{F}}_{\mathcal{A}}}=
\|E(x^*x)\|_{\mathcal{A}}^{\frac{1}{2}}$ and $e_{\mathcal{A}}$ is the Jones projection of $E$.

We will give some properties about the elements in $L_{\mathcal{A}}(\overline{\mathcal{F}})$ as follows.

\begin{proposition}
(1)\ As operators on $\overline{\mathcal{F}}$, we have $e_{\mathcal{A}}Te_{\mathcal{A}}=E(T)e_{\mathcal{A}}$.

(2)\ Let $T\in {\mathcal{F}}$, then
$T\in {\mathcal{A}}$ if and only if $e_{\mathcal{A}}T=Te_{\mathcal{A}}$.
\end{proposition}

\section{$C^*$-isomorphism between ${\mathcal{F}}\rtimes D(G)$ and $\langle{\mathcal{F}},e_{\mathcal{A}}\rangle_{C^*}$}
In this section, we will construct the crossed product $C^*$-algebra
${\mathcal{F}}\rtimes D(G)$ extending ${\mathcal{F}}\equiv{\mathcal{F}}\rtimes I_{D(G)}$ by means of a
Hopf module left action of $D(G)$ on ${\mathcal{F}}$, such that ${\mathcal{F}}\rtimes D(G)$ coincides with the $C^*$-algebra $\langle{\mathcal{F}},e_{\mathcal{A}}\rangle_{C^*}$ constructed from the $C^*$-basic construction for the inclusion ${\mathcal{A}}\subseteq{\mathcal{F}}$.

As we have known, the field algebra ${\mathcal{F}}$ of $G$-spin models is a $D(G)$-module algebra, and one can construct the crossed product $*$-algebra ${\mathcal{F}}_{\mathrm{loc}}\rtimes D(G)$, as a vector space ${\mathcal{F}}_{\mathrm{loc}}\otimes D(G)$ with the $*$-algebra structure
\begin{eqnarray*}
    \begin{array}{rcl}
  \Big(T\otimes (g,h)\Big)\Big(F\otimes (s,t)\Big)
  &=& \sum\limits_{(g,h)}T(g,h)_{(1)}(F)\otimes(g,h)_{(2)}(s,t),\\[5pt]
  \Big(T\otimes (g,h)\Big)^*&=&\Big(I_{\mathcal{F}}\otimes (g,h)^*\Big)\Big(T^*\otimes I_{D(G)}\Big).
       \end{array}
\end{eqnarray*}
Using the $C^*$-inductive limit \cite{B.R.Li}, the crossed product ${\mathcal{F}}_{\mathrm{loc}}\rtimes D(G)$ can be extended naturally to a $C^*$-algebra ${\mathcal{F}}\rtimes D(G)$.

Now, we consider a special element $I_{\mathcal{F}}\rtimes \frac{1}{|G|}\sum\limits_{g\in G}(u,g)$ in ${\mathcal{F}}\rtimes D(G)$.

\begin{proposition}
The element $I_{\mathcal{F}}\rtimes \frac{1}{|G|}\sum\limits_{g\in G}(u,g)$ is a self-adjoint idempotent element. That is
\begin{eqnarray*}
 \begin{array}{c}
\Big(I_{\mathcal{F}}\rtimes \frac{1}{|G|}\sum\limits_{g\in G}(u,g)\Big)^2\ =\
I_{\mathcal{F}}\rtimes \frac{1}{|G|}\sum\limits_{g\in G}(u,g) \ =\
\Big(I_{\mathcal{F}}\rtimes \frac{1}{|G|}\sum\limits_{g\in G}(u,g)\Big)^*.
       \end{array}
\end{eqnarray*}
\end{proposition}
\begin{proof}
We can compute that
\begin{eqnarray*}
 \begin{array}{rcl}
&&\Big(I_{\mathcal{F}}\rtimes \frac{1}{|G|}\sum\limits_{g\in G}(u,g)\Big)^2\\[5pt]
&=&\Big(\sum\limits_{g_1,g_2\in G}\delta_{g_1}(1)\delta_{g_2}(2)\rtimes
 \frac{1}{|G|}\sum\limits_{h\in G}(u,h) \Big)
 \Big(\sum\limits_{s_1,s_2\in G}\delta_{s_1}(1)\delta_{s_2}(2)\rtimes
 \frac{1}{|G|}\sum\limits_{t\in G}(u,t) \Big)\\[5pt]
&=&\frac{1}{|G|^2}\delta_{g_1,hs_1}\delta_{g_2,hs_2}
\sum\limits_{g_1,g_2,s_1,s_2,h,t\in G}
\delta_{g_1}(1)\delta_{g_2}(2)\rtimes (u,ht)\\[5pt]
&=&\frac{1}{|G|^2}\sum\limits_{g_1,g_2,h,t\in G}
\delta_{g_1}(1)\delta_{g_2}(2)\rtimes (u,ht)\\[5pt]
&=&\frac{1}{|G|}\Big(I_{\mathcal{F}}\rtimes \sum\limits_{h\in G}(u,h)\Big),
        \end{array}
\end{eqnarray*}
and
\begin{eqnarray*}
 \begin{array}{rcl}
\Big(I_{\mathcal{F}}\rtimes \frac{1}{|G|}\sum\limits_{g\in G}(u,g)\Big)^*
&=&\frac{1}{|G|}\Big(I_{\mathcal{F}}\rtimes\sum\limits_{g\in G}(u,g)^*\Big)
\Big(I_{\mathcal{F}}\rtimes I_{D(G)}\Big)\\[5pt]
&=&\frac{1}{|G|}\Big(I_{\mathcal{F}}\rtimes\sum\limits_{g\in G}(u,g^{-1})\Big)
\Big(I_{\mathcal{F}}\rtimes I_{D(G)}\Big)\\[5pt]
&=&\frac{1}{|G|}\Big(I_{\mathcal{F}}\rtimes\sum\limits_{g\in G}(u,g)\Big).
        \end{array}
\end{eqnarray*}
Hence, $I_{\mathcal{F}}\rtimes \frac{1}{|G|}\sum\limits_{g\in G}(u,g)$ is a self-adjoint idempotent element.
\end{proof}

\begin{proposition}
The element
$T\rtimes I_{D(G)}$ in ${\mathcal{F}}\rtimes D(G)$ satisfies the following covariant relation

$\Big(I_{\mathcal{F}}\rtimes \frac{1}{|G|}\sum\limits_{g\in G}(u,g)\Big)\Big(T\rtimes I_{D(G)}\Big)\Big(I_{\mathcal{F}}\rtimes \frac{1}{|G|}\sum\limits_{g\in G}(u,g)\Big)
=\Big(E(T)\rtimes I_{D(G)}\Big)\Big(I_{\mathcal{F}}\rtimes \frac{1}{|G|}\sum\limits_{g\in G}(u,g)\Big)$.
\end{proposition}

\begin{proof}
Suppose that $T=\delta_{s_1}(1)\delta_{s_2}(2)
\rho_{t_1}(\frac{1}{2})\rho_{t_2}(\frac{3}{2})$, we can obtain
\begin{eqnarray*}
 \begin{array}{rcl}
&&\Big(I_{\mathcal{F}}\rtimes \frac{1}{|G|}\sum\limits_{g\in G}(u,g)\Big)\Big(T\rtimes I_{D(G)}\Big)\\[5pt]
&=&\frac{1}{|G|}\Big(\sum\limits_{g_1,g_2\in G} \delta_{g_1}(1)\delta_{g_2}(2)\rtimes\sum\limits_{g\in G}(u,g) \Big)\Big(\delta_{s_1}(1)\delta_{s_2}(2)
\rho_{t_1}(\frac{1}{2})\rho_{t_2}(\frac{3}{2})\rtimes\sum\limits_{s\in G}(s,u)\Big)\\[5pt]
&=&\frac{1}{|G|}\sum\limits_{g_1,g_2,g,s\in G}
\delta_{u,t_1t_2s}\delta_{g_1,gs_1}\delta_{g_2,gs_2}
\delta_{g_1}(1)\delta_{g_2}(2)\rho_{gt_1g^{-1}}(\frac{1}{2})
\rho_{gt_2g^{-1}}(\frac{3}{2})\rtimes(gsg^{-1},g)\\[5pt]
&=&\frac{1}{|G|}\sum\limits_{g\in G}
\delta_{gs_1}(1)\delta_{gs_2}(2)\rho_{gt_1g^{-1}}(\frac{1}{2})
\rho_{gt_2g^{-1}}(\frac{3}{2})\rtimes(gt_2^{-1}t_1^{-1}g^{-1},g),
       \end{array}
\end{eqnarray*}
and then
\begin{eqnarray*}
 \begin{array}{rcl}
&&\Big(I_{\mathcal{F}}\rtimes \frac{1}{|G|}\sum\limits_{g\in G}(u,g)\Big)\Big(T\rtimes I_{D(G)}\Big)\Big(I_{\mathcal{F}}\rtimes \frac{1}{|G|}\sum\limits_{g\in G}(u,g)\Big)\\[5pt]
&=&\frac{1}{|G|}\sum\limits_{g\in G}
\delta_{gs_1}(1)\delta_{gs_2}(2)\rho_{gt_1g^{-1}}(\frac{1}{2})
\rho_{gt_2g^{-1}}(\frac{3}{2})\rtimes(gt_2^{-1}t_1^{-1}g^{-1},g)
(\sum\limits_{g_1,g_2\in G}\delta_{g_1}(1)\delta_{g_2}(2)\rtimes\frac{1}{|G|}\sum\limits_{g\in G}(u,g) )\\[5pt]
&=&\frac{1}{|G|^2}\sum\limits_{g,g_1,g_2,f\in G}
\delta_{t_1t_2,u}\delta_{s_1,t_1g_1}\delta_{s_2,g_2}
\delta_{gs_1}(1)\delta_{gs_2}(2)\rho_{gt_1g^{-1}}(\frac{1}{2})
\rho_{gt_2g^{-1}}(\frac{3}{2})\rtimes(u,gf)\\[5pt]
&=&\frac{1}{|G|^2}\delta_{t_1t_2,u}\sum\limits_{g,f\in G}
\delta_{gs_1}(1)\delta_{gs_2}(2)\rho_{gt_1g^{-1}}(\frac{1}{2})
\rho_{gt_2g^{-1}}(\frac{3}{2})\rtimes(u,gf)\\[5pt]
&=&\frac{1}{|G|^2}\delta_{t_1t_2,u}\sum\limits_{g,f\in G}
\delta_{gs_1}(1)\delta_{gs_2}(2)\rho_{gt_1g^{-1}}(\frac{1}{2})
\rho_{gt_2g^{-1}}(\frac{3}{2})\rtimes(u,f).
       \end{array}
\end{eqnarray*}
On the other hand,
\begin{eqnarray*}
 \begin{array}{rcl}
&& E(T)=E\Big(\delta_{s_1}(1)\delta_{s_2}(2)
\rho_{t_1}(\frac{1}{2})\rho_{t_2}(\frac{3}{2})\Big)\\[5pt]
&=&\frac{1}{|G|}\delta_{t_1t_2,u}\sum\limits_{f\in G}\delta_{fs_1}(1)\delta_{fs_2}(2)
\rho_{ft_1f^{-1}}(\frac{1}{2})\rho_{ft_2f^{-1}}(\frac{3}{2}),
        \end{array}
\end{eqnarray*}
and then
\begin{eqnarray*}
 \begin{array}{rcl}
&&\Big(E(T)\rtimes I_{D(G)}\Big)\Big(I_{\mathcal{F}}\rtimes \frac{1}{|G|}\sum\limits_{g\in G}(u,g)\Big)\\[5pt]
&=&\Big(\frac{1}{|G|}\delta_{t_1t_2,u}\sum\limits_{f\in G}\delta_{fs_1}(1)\delta_{fs_2}(2)
\rho_{ft_1f^{-1}}(\frac{1}{2})\rho_{ft_2f^{-1}}(\frac{3}{2})\rtimes
 \sum\limits_{s\in G}(s,u)\Big)\Big(\sum\limits_{g_1,g_2\in G} \delta_{g_1}(1)\delta_{g_2}(2)\rtimes \frac{1}{|G|}\sum\limits_{t\in G}(u,t)\Big)\\[5pt]
&=&\frac{1}{|G|^2}\delta_{t_1t_2,u}
\sum\limits_{f,t,g_1,g_2,s\in G}
\delta_{s,u}\delta_{s_1,t_1f^{-1}g_1}\delta_{fs_2,g_2}
\delta_{fs_1}(1)\delta_{fs_2}(2)
\rho_{ft_1f^{-1}}(\frac{1}{2})\rho_{ft_2f^{-1}}(\frac{3}{2})
\rtimes(u,t)\\[5pt]
&=&\frac{1}{|G|^2}\delta_{t_1t_2,u}
\sum\limits_{f,t\in G}
\delta_{fs_1}(1)\delta_{fs_2}(2)
\rho_{ft_1f^{-1}}(\frac{1}{2})\rho_{ft_2f^{-1}}(\frac{3}{2})
\rtimes(u,t).
        \end{array}
\end{eqnarray*}
From the above, we can obtain the desired result.
\end{proof}

The following theorem is one of main results of this paper, which gives a characterization of the $C^*$-algebra $\langle{\mathcal{F}},e_{\mathcal{A}}\rangle_{C^*}$ constructed from the $C^*$-basic construction for the inclusion ${\mathcal{A}}\subseteq {\mathcal{F}}$.

\begin{theorem}
There exists a $C^*$-isomorphism
between the crossed product $C^*$-algebra ${\mathcal{F}}\rtimes D(G)$ and
the $C^*$-algebra $\langle{\mathcal{F}},e_{\mathcal{A}}\rangle_{C^*}$. That is,
$${\mathcal{F}}\rtimes D(G)\cong \langle{\mathcal{F}},e_{\mathcal{A}}\rangle_{C^*}.$$
\end{theorem}
\begin{proof}
 We first show that the action of $D(G)$ on $\mathcal{F}$ is faithful, that is,
  $(g,h)F=(s,t)F$ for any $F\in{\mathcal{F}}$ implies $(g,h)=(s,t)$.
To this end let $F=\sum\limits_{f\in G}\delta_{f}(x)$, then $(g,h)(\sum\limits_{f\in G}\delta_{f}(x))=(s,t)(\sum\limits_{f\in G}\delta_{f}(x))$,
and therefore
\begin{eqnarray*}
    \begin{array}{c}
    \sum\limits_{f\in G}\delta_{g,u}\delta_{hf}(x)
=\sum\limits_{f\in G}\delta_{s,u}\delta_{tf}(x).
\end{array}
\end{eqnarray*}
Using $\sum\limits_{f\in G}\delta_{hf}(x)=I$, we have $\delta_{g,u}=\delta_{s,u}$ and then
$g=s$.
Let now $F=\delta_{u}(x)\rho_{h^{-1}gh}(l)$, then we conclude $(g,h)(\delta_{u}(x)\rho_{h^{-1}gh}(l))=(s,t)(\delta_{u}(x)\rho_{h^{-1}gh}(l))$,
which proves
\begin{eqnarray*}
    \begin{array}{rcl}
\sum\limits_{a\in G}(a,h)\delta_{u}(x)(a^{-1}g,h)\rho_{h^{-1}gh}(l)
 &=& \sum\limits_{b\in G}(b,t)\delta_{u}(x)(b^{-1}g,t)\rho_{h^{-1}gh}(l), \\[5pt]
 \sum\limits_{a\in G}\delta_{a,u}\delta_{h}(x)
 \delta_{a^{-1}g,g}\rho_{a^{-1}g}(l)
 &=& \sum\limits_{b\in G}\delta_{b,u}\delta_{t}(x)
 \delta_{b^{-1}g,th^{-1}ght^{-1}}\rho_{b^{-1}g}(l).
 \end{array}
\end{eqnarray*}
Putting $a=u=b$, we get
$$\delta_{h}(x)\rho_{g}(l)= \delta_{g,th^{-1}ght^{-1}}\delta_{t}(x)\rho_{g}(l),$$
and therefore $h=t$.
Hence we get the action of ${\mathcal{F}}\rtimes D(G)$ on $\mathcal{F}$ is also faithful.

Secondly, we have proved in \cite{XINQL} that for a fixed $k\in\Bbb{Z}$, the set \begin{eqnarray*}
    \begin{array}{c}
    \{|G|^{1/2}\delta_g(k)\rho_h(k+\frac{1}{2}),|G|^{1/2}(\delta_g(k)\rho_h(k+\frac{1}{2}))^*: g,h\in G\}
 \end{array}
\end{eqnarray*}
is a quasi basis of $\mathcal{F}$ over $\mathcal{A}$, implying the index is finite type in the sense of Watatani. It then follows from Propositon 1.3.3 in \cite{Y.Wat} that $\langle{\mathcal{F}},e_{\mathcal{A}}\rangle_{C^*}$ is the same as $B({\mathcal{F}})$, the algebra of bounded, adjointable operators in ${\mathcal{F}}_{\mathcal{A}}$ (with the standard ${\mathcal{A}}$-valued inner product).

Finally, the element $e_{\mathcal{A}}$ is represented by an element in $D(G)$,
so combined with the previous bullet point, this yields
$C^*$-algebra $\langle {\mathcal{F}},e_{\mathcal{A}}\rangle_{C^*}$
coincides with the crossed product $C^*$-algebra ${\mathcal{F}}\rtimes D(G)$.
\end{proof}

\begin{remark}
From Theorem 3.1, we know that the $C^*$-basic constructions do not depend on the choice of conditional expectations, which can also be seen in Proposition 2.10.11 \cite{Y.Wat}.
\end{remark}

\section{The $C^*$-basic construction for the inclusion ${\mathcal{F}}\subseteq {\mathcal{F}}\rtimes D(G)$}

In this section, we continue to investigate the crossed
product $C^*$-algebra ${\mathcal{F}}\rtimes D(G)$, and the natural $\widehat{D(G)}$-module algebra action on ${\mathcal{F}}\rtimes D(G)$, which gives rise to the iterated crossed product $C^*$-algebra ${\mathcal{F}}\rtimes D(G)\rtimes\widehat{D(G)}$.
The fixed point algebra under this action is given by ${\mathcal{F}}\equiv{\mathcal{F}}\rtimes I{_D(G)}$, which is consistent with
the range of the conditional expectation $E_2$.
We then prove that the $C^*$-algebra $\langle{\mathcal{F}}\rtimes D(G),e_2\rangle_{C^*}$ constructed from the
$C^*$-basic construction
for the inclusion
${\mathcal{F}}\subseteq{\mathcal{F}}\rtimes D(G)$ is precisely $C^*$-isomorphic to the iterated crossed product $C^*$-algebra ${\mathcal{F}}\rtimes D(G)\rtimes\widehat{D(G)}$.

Since $D(G)$ is of finite dimension and $\widehat{{C(G)\otimes \Bbb{C}G}}\cong \Bbb{C}G\otimes C(G)$ as algebras, $\{(y,\delta_x)\colon y,x\in G\}$ can be viewed as a linear basis of $\widehat{D(G)}$. As the above states, the structure maps on $\widehat{D(G)}$ are the following
\begin{eqnarray*}
 \begin{array}{rcll}
   \widetilde{\bigtriangleup}(y,\delta_x)
&=&\sum \limits_{t\in G}(y,\delta_{t^{-1}})\otimes(tyt^{-1},\delta_{tx}),
&(\mbox{coproduct})\\
   (y,\delta_x)(w,\delta_z)
&=&\delta_{x,z}(yw,\delta_x),&(\mbox{multiplication})\\[5pt]
   (y,\delta_x)^*
&=&(y^{-1},\delta_x),&(\mbox{$*$-operation})\\[5pt]
    \widetilde{\varepsilon}(y,\delta_x)
&=&\delta_{x,u},&(\mbox{counit})\\[5pt]
   \widetilde{S}(y,\delta_x)
&=&(x^{-1}y^{-1}x,\delta_{x^{-1}}).&(\mbox{antipode})
 \end{array}
\end{eqnarray*}
It is easy to see that $I_{\widehat{D(G)}}=
\sum\limits_{x\in G}(u,\delta_x)$, and there is a unique element $E_2=\frac{1}{|G|}\sum\limits_{y\in G}(y,\delta_u)$, called an integral element, satisfying for any $b\in \widehat{D(G)}$,
$$bE_2=E_2b=\varepsilon(b)E_2.$$

The map $\sigma\colon \widehat{D(G)}\times ({\mathcal{F}}\rtimes D(G))\rightarrow {\mathcal{F}}\rtimes D(G)$ given on the generating elements of ${\mathcal{F}}\rtimes D(G)$ as
$$\sigma\Big((y,\delta_x)\times(F\otimes (g,h))\Big)=\delta_{x,h}\Big(F\otimes(gy^{-1},h)\Big)$$
for $(g,h)\in D(G)$, can be linearly extended both in $\widehat{D(G)}$ and ${\mathcal{F}}\rtimes D(G)$.
Here and from now on, by $(y,\delta_x)(F\otimes (g,h))$ we always denote $\sigma((y,\delta_x)\times(F\otimes (g,h)))$ for notational
convenience.

In particular, considering the action of $E_2$ on ${\mathcal{F}}\rtimes D(G)$, we can obtain that
\begin{eqnarray*}
 \begin{array}{c}
E_2(F\otimes (g,h))
 =\frac{1}{|G|}\sum\limits_{y\in G}(y,\delta_u)(F\otimes (g,h))
 =\frac{1}{|G|}\sum\limits_{y\in G}\delta_{h,u}(F\otimes(gy^{-1},h))
 =\frac{1}{|G|}\delta_{h,u}(F\otimes I_{D(G)}),
\end{array}
\end{eqnarray*}
which means that the range of $E_2$ on ${\mathcal{F}}\rtimes D(G)$ is contained in ${\mathcal{F}}$. Moreover, we can show that
$E_2$ is a positive map preserving the unit and possessing the bimodular property. Namely,

\begin{proposition}
The map $E_2\colon {\mathcal{F}}\rtimes D(G)\rightarrow {\mathcal{F}}$ is a conditional expectation.
\end{proposition}
\begin{proof}
(1)\ $E_2(I_{{\mathcal{F}}\rtimes D(G)})=E_2(I_{\mathcal{F}}\otimes \sum\limits_{g\in G}(g,u))=\frac{1}{|G|}\sum\limits_{g\in G}(I_{\mathcal{F}}\otimes I_{D(G)})=I_{\mathcal{F}}.$

 (2)\ $\forall T_1, T_2\in{\mathcal{F}}, \widetilde{T}\in {\mathcal{F}}\rtimes D(G)$, we have
 \begin{eqnarray*}
    \begin{array}{rcl}
 E_2(T_1\widetilde{T}T_2)
 &=&\frac{1}{|G|}\sum\limits_{y\in G}(y,\delta_u)(T_1\widetilde{T}T_2)\\[5pt]
 &=&\frac{1}{|G|}\sum\limits_{y,t_1,t_2\in G}(y,\delta_{t_1^{-1}})(T_1)(t_1yt_1^{-1},\delta_{t_1t_2^{-1}})
 (\widetilde{T})
 (t_2yt_2^{-1},\delta_{t_2})(T_2)\\[5pt]
 &=&\frac{1}{|G|}\sum\limits_{y,t_1,t_2\in G}
 \widetilde{\varepsilon}(y,\delta_{t_1^{-1}})(T_1)
 (t_1yt_1^{-1},\delta_{t_1t_2^{-1}})(\widetilde{T})
 \widetilde{\varepsilon}(t_2yt_2^{-1},\delta_{t_2})(T_2)\\[5pt]
 &=&\frac{1}{|G|}\sum\limits_{y,t_1,t_2\in G}
 \delta_{t_1^{-1},u}(T_1)
 (t_1yt_1^{-1},\delta_{t_1t_2^{-1}})(\widetilde{T})
 \delta_{t_2,u}(T_2)\\[5pt]
 &=&\frac{1}{|G|}\sum\limits_{y\in G}T_1(y,\delta_u)(\widetilde{T})T_2\\[5pt]
 &=&T_1E_2(\widetilde{T})T_2.
 \end{array}
\end{eqnarray*}

(3)\ We note the relation for any
 $\widetilde{T}\in {\mathcal{F}}\rtimes D(G)$,
 \begin{eqnarray*}
    \begin{array}{rcl}
    E_2(\widetilde{T}^*\widetilde{T})
 &=&\frac{1}{|G|}\sum\limits_{y\in G}(y,\delta_u)(\widetilde{T}^*\widetilde{T})\\[5pt]
 &=&\frac{1}{|G|}\sum\limits_{y,t\in G}(y,\delta_{t^{-1}})(\widetilde{T}^*)
 (tyt^{-1},\delta_t)(\widetilde{T})\\[5pt]
 &=&\frac{1}{|G|}\sum\limits_{y,t\in G}(\widetilde{S}(y,\delta_{t^{-1}})^*
 (\widetilde{T}))^*(tyt^{-1},\delta_t)(\widetilde{T})\\[5pt]
 &=&\frac{1}{|G|}\sum\limits_{y,t\in G}
 ((tyt^{-1},\delta_t)(F))^*(tyt^{-1},\delta_t)(\widetilde{T}),
  \end{array}
\end{eqnarray*}
which means $E_2$ is a positive map on ${\mathcal{F}}\rtimes D(G)$.
\end{proof}

\begin{proposition}
The map $\sigma$ defines a Hopf module left action of $\widehat{D(G)}$ on ${\mathcal{F}}\rtimes D(G)$. That is ${\mathcal{F}}\rtimes D(G)$ is a left $\widehat{D(G)}$-module algebra.
\end{proposition}

\begin{proof}
It suffices to check that the map $\sigma\colon \widehat{D(G)}\times ({\mathcal{F}}\rtimes D(G))\rightarrow {\mathcal{F}}\rtimes D(G)$ satisfies the following relations:
\begin{eqnarray*}
    \begin{array}{rcl}
\Big((y,\delta_x)(w,\delta_z)\Big)\Big(F\otimes(g,h)\Big)&=&
(y,\delta_x)\Big((w,\delta_z)(F\otimes(g,h))\Big),\\[5pt]
    (y,\delta_x)\Big((F\otimes(g_1,h_1))(T\otimes(g_2,h_2))\Big)&=&
\sum\limits_{(y,\delta_x)}\Big((y,\delta_x)_{(1)}(F\otimes(g_1,h_1))\Big)
\Big((y,\delta_x)_{(2)}(T\otimes(g_2,h_2))\Big),\\[5pt]
    (y,\delta_x)\Big(F\otimes(g,h)\Big)^*&=&
\Big(\widetilde{S}(y,\delta_x)^*(F\otimes(g,h))\Big)^*,
    \end{array}
\end{eqnarray*}
 for $(y,\delta_x),(w,\delta_z)\in \widehat{D(G)}$, $T,F\in{\mathcal{F}}$ and $(g_i,h_i),(g,h)\in D(G)$ for $i=1,2$.

As to the first equality, we compute
\begin{eqnarray*}
    \begin{array}{clcl}
    \Big((y,\delta_x)(w,\delta_z)\Big)(F\otimes(g,h))
&=&\delta_{x,z}(yw,\delta_x)(F\otimes(g,h))\\[5pt]
&=&\delta_{x,z}\delta_{x,h}(F\otimes(gw^{-1}y^{-1},h))\\[5pt]
&=&\delta_{z,h}\delta_{x,h}(F\otimes(gw^{-1}y^{-1},h))\\[5pt]
&=&\delta_{z,h}(y,\delta_x)(F\otimes(gw^{-1},h))\\[5pt]
&=&(y,\delta_x)\Big((w,\delta_z)(F\otimes(g,h))\Big).
    \end{array}
\end{eqnarray*}

Next,
\begin{eqnarray*}
    \begin{array}{rcl}
  && (y,\delta_x)\Big((F\otimes(g_1,h_1))(T\otimes(g_2,h_2))\Big)\\[5pt]
&=&\sum\limits_{(g_1,h_1)}(y,\delta_x)
 \Big(F(g_1,h_1)_{(1)}T\otimes(g_1,h_1)_{(2)}(g_2,h_2)\Big)\\[5pt]
&=&\sum\limits_{f\in G}(y,\delta_x)
 \Big(F(f,h_1)T\otimes(f^{-1}g_1,h_1)(g_2,h_2)\Big)\\[5pt]
&=&\sum\limits_{f\in G}(y,\delta_x)
 \Big(F(f,h_1)T\otimes\delta_{f^{-1}g_1h_1,h_1g_2}(f^{-1}g_1,h_1h_2)\Big)\\[5pt]
&=&(y,\delta_x)
 \Big(F(g_1h_1g_2^{-1}h_1^{-1},h_1)T\otimes(h_1g_2h_1^{-1},h_1h_2)\Big)\\[5pt]
&=&\delta_{x,h_1h_2}F(g_1h_1g_2^{-1}h_1^{-1},h_1)T\otimes
(h_1g_2h_1^{-1}y^{-1},h_1h_2),
    \end{array}
\end{eqnarray*}
and
\begin{eqnarray*}
    \begin{array}{rcl}
   &&\sum\limits_{(y,\delta_x)}\Big((y,\delta_x)_{(1)}
   (F\otimes(g_1,h_1))\Big)
\Big((y,\delta_x)_{(2)}(T\otimes(g_2,h_2))\Big)\\[5pt]
&=&\sum\limits_{t\in G}\Big((y,\delta_{t^{-1}})(F\otimes(g_1,h_1))\Big)
\Big((tyt^{-1},\delta_{tx})(T\otimes(g_2,h_2))\Big)\\[5pt]
&=&\sum\limits_{t\in G}\delta_{t^{-1},h_1}\delta_{tx,h_2}
\Big(F\otimes(g_1y^{-1},h_1)\Big)\Big(T\otimes(g_2ty^{-1}t^{-1},h_2)\Big)\\[5pt]
&=&\delta_{x,h_1h_2}
\Big(F\otimes(g_1y^{-1},h_1)\Big)
\Big(T\otimes(g_2h_1^{-1}y^{-1}h_1,h_2)\Big)\\[5pt]
&=&\sum\limits_{f\in G}\delta_{x,h_1h_2}
F(f,h_1)T\otimes(f^{-1}g_1y^{-1},h_1)(g_2h_1^{-1}y^{-1}h_1,h_2)\\[5pt]
&=&\delta_{x,h_1h_2}\delta_{f^{-1}g_1,h_1g_2h_1^{-1}}
F(f,h_1)T\otimes(f^{-1}g_1y^{-1},h_1h_2)\\[5pt]
&=&\delta_{x,h_1h_2}F(g_1h_1g_2^{-1}h_1^{-1},h_1)T\otimes
(h_1g_2h_1^{-1}y^{-1},h_1h_2).
    \end{array}
\end{eqnarray*}
Thus, we obtain that
\begin{eqnarray*}
    \begin{array}{c}
       (y,\delta_x)\Big((F\otimes(g_1,h_1))(T\otimes(g_2,h_2))\Big)=
\sum\limits_{(y,\delta_x)}\Big((y,\delta_x)_{(1)}(F\otimes(g_1,h_1))\Big)
\Big((y,\delta_x)_{(2)}(T\otimes(g_2,h_2))\Big).
    \end{array}
\end{eqnarray*}

To prove the third equation, we can calculate
\begin{eqnarray*}
    \begin{array}{rcl}
(y,\delta_x)\Big(F\otimes(g,h)\Big)^*&=&(y,\delta_x)\Big((I_{\mathcal{F}}
\otimes(g,h)^*)(F^*\otimes I_{D(G)})\Big)\\[5pt]
&=&(y,\delta_x)\Big((I_{\mathcal{F}}
\otimes(h^{-1}gh,h^{-1}))(F^*\otimes I_{D(G)})\Big)\\[5pt]
&=&\sum\limits_{f\in G}
(y,\delta_x)\Big((f,h^{-1})F^*\otimes(f^{-1}h^{-1}gh,h^{-1})\Big)\\[5pt]
&=&\sum\limits_{f\in G}\delta_{x,h^{-1}}
(f,h^{-1})F^*\otimes(f^{-1}h^{-1}ghy^{-1},h^{-1}),
    \end{array}
\end{eqnarray*}
and
\begin{eqnarray*}
    \begin{array}{rcl}
\Big(\widetilde{S}(y,\delta_x)(F\otimes(g,h))\Big)^*
&=&\Big((x^{-1}yx,\delta_{x^{-1}})(F\otimes(g,h))\Big)^*\\[5pt]
&=&\delta_{x^{-1},h}\Big(F\otimes(gx^{-1}y^{-1}x,h)\Big)^*\\[5pt]
&=&\sum\limits_{f\in G}\delta_{x,h^{-1}}
(f,h^{-1})F^*\otimes(f^{-1}h^{-1}ghy^{-1},h^{-1}).
    \end{array}
\end{eqnarray*}
\end{proof}

From Proposition 4.2, we can construct the crossed product $({\mathcal{F}}\rtimes D(G))\rtimes\widehat{D(G)}$,
which is called the iterated crossed product $C^*$-algebra.

In the following, we will consider the $\widehat{D(G)}$-invariant subalgebra of ${\mathcal{F}}\rtimes D(G)$. To do this, set
$$({\mathcal{F}}\rtimes D(G))^{\widehat{D(G)}}\ =\
\{\widetilde{T}\in {\mathcal{F}}\rtimes D(G)\colon b(\widetilde{T})=\widetilde{\varepsilon}(b)(\widetilde{T}), \ \ \forall b\in \widehat{D(G)}\}.$$
 One can show that $({\mathcal{F}}\rtimes D(G))^{\widehat{D(G)}}$ is a $C^*$-subalgebra of ${\mathcal{F}}\rtimes D(G)$. Furthermore, $$({\mathcal{F}}\rtimes D(G))^{\widehat{D(G)}}\ =\ \{ \widetilde{T}\in {\mathcal{F}}\rtimes D(G)\colon E_2(\widetilde{T})=\widetilde{T}\}.$$
In fact, for $\widetilde{T}\in ({\mathcal{F}}\rtimes D(G))^{\widehat{D(G)}}\subseteq {\mathcal{F}}\rtimes D(G)$, we can compute that
\begin{eqnarray*}
    \begin{array}{c}
    E_2(\widetilde{T})=\frac{1}{|G|}\sum\limits_{y\in G}(y,\delta_u)
    (\widetilde{T})
    =\frac{1}{|G|}\sum\limits_{y\in G}\widetilde{\varepsilon}(y,\delta_u)(\widetilde{T})
    =\frac{1}{|G|}\sum\limits_{y\in G}\widetilde{T}
    =\widetilde{T}.
   \end{array}
\end{eqnarray*}
For the converse, suppose that $ \widetilde{T}\in {\mathcal{F}}\rtimes D(G)$ with $E_2(\widetilde{T})=\widetilde{T}$. Then for any $(w,\delta_z)\in \widehat{D(G)}$, we have
\begin{eqnarray*}
    \begin{array}{rcl}
    (w,\delta_z)(\widetilde{T})
    &=&(w,\delta_z)E_2(\widetilde{T})\\[5pt]
    &=&(w,\delta_z)\frac{1}{|G|}\sum\limits_{y\in G}(y,\delta_u)
    (\widetilde{T})\\[5pt]
    &=&\frac{1}{|G|}\sum\limits_{y\in G}\delta_{z,u}(wy,\delta_u)(\widetilde{T})\\[5pt]
    &=&\widetilde{\varepsilon}(w,\delta_z)E_2(\widetilde{T})\\[5pt]
    &=&\widetilde{\varepsilon}(w,\delta_z)(\widetilde{T}).
   \end{array}
\end{eqnarray*}
This can be linearly extended in $\widehat{D(G)}$. Hence, $\widetilde{T}\in ({\mathcal{F}}\rtimes D(G))^{\widehat{D(G)}}.$

\begin{remark}

$({\mathcal{F}}\rtimes D(G))^{\widehat{D(G)}}$ is the subalgebra of ${\mathcal{F}}\rtimes D(G)$ corresponding to the trivial representation $\widetilde{\varepsilon}$ of $\widehat{D(G)}$. 
\end{remark}

Naturally, we consider the $C^*$-algebra $\langle{\mathcal{F}}\rtimes D(G),e_2\rangle_{C^*}$ constructed from the
$C^*$-basic construction for the inclusion ${\mathcal{F}}\subseteq{\mathcal{F}}\rtimes D(G)$
 in the following, where $e_2$ is the Jones projection of $E_2$.

\begin{theorem}
There exists a $C^*$-isomorphism
of $C^*$-algebras between ${\mathcal{F}}\rtimes D(G)\rtimes \widehat{D(G)}$
and $\langle{\mathcal{F}}\rtimes D(G),e_2\rangle_{C^*}$. That is,
$${\mathcal{F}}\rtimes D(G)\rtimes \widehat{D(G)}\cong \langle{\mathcal{F}}\rtimes D(G),e_2\rangle_{C^*}.$$
\end{theorem}
\begin{proof}
The proof is similar to that of Theorem 3.1.
\end{proof}

Moreover, by Takai duality \cite{H.Tak}, the iterated crossed product $C^*$-algebra ${\mathcal{F}}\rtimes D(G)\rtimes \widehat{D(G)}$
is canonically isomorphic to $M_{|G|^2}({\mathcal{F}})$.

\begin{remark}
In the following we will give the concrete construction for $M_{|G|^2}({\mathcal{F}})$.

The local field $M_{|G|^2}({\mathcal{F}}_{\mathrm{loc}})$ of a $M_{|G|^2}(G)$-spin model is a $*$-algebra with a unit $I_{M_{|G|^2}({\mathcal{F}})}$ generated by
 $\Big\{\delta_g(x)\otimes M, \rho_h(l)\otimes N\colon g, h\in G, \ x\in {\Bbb Z},\ l\in {\Bbb Z}+\frac{1}{2},\ M, N\in \mbox{LB} (M_{|G|^2})\Big\}$ satisfying the following relations
\begin{eqnarray*}
    \begin{array}{rcl}
O_M^g(x)O_N^h(x)&=&\delta_{g,h}O_{MN}^g(x),\\[5pt]
 D_M^g(l)D_N^h(l)&=&D_{MN}^{gh}(l),\\[5pt]
 \sum\limits_{g\in G}
O_I^g(x)&=&I_{M_{|G|^2}({\mathcal{F}})}
=D_I^u(l),\\[5pt]
O_M^g(x)O_N^h(x')&=&O_M^h(x')O_N^g(x),\\[5pt]
 D_M^g(l)O_N^h(x)&=& \left\{\begin{array}{cc}
                          O_M^{gh}(x)D_N^g(l), & \mbox{if}\ l<x, \\[5pt]
                          O_M^h(x)D_N^g(l), & \mbox{if}\ l>x,
                        \end{array}\right.
    \end{array}
\end{eqnarray*}
\begin{eqnarray*}
    \begin{array}{rcl}
 D_M^g(l)D_N^h(l')&=& \left\{\begin{array}{cc}
                          D_M^h(l')D_N^{h^{-1}gh}(l), & \mbox{if}\ l>l',\\[5pt]
                          D_M^{ghg^{-1}}(l')D_N^{g}(l), & \mbox{if}\ l<l',
                        \end{array}\right.\\[5pt]
 (O_M^g(x))^*&=&O_{M^*}^g(x),\\[5pt]
  (D_N^h(l))^*&=&D_{N^*}^{h^{-1}}(l),
    \end{array}
\end{eqnarray*}
for $x, x' \in {\Bbb Z},\ l, l'\in {\Bbb Z}+\frac{1}{2} \ {\mbox{and}} \ g, h\in G$, where by $O_M^g(x)$, $D_N^h(l)$ and $\mbox{LB}(M_{|G|^2})$ we denote $\delta_g(x)\otimes M$, $\rho_h(l)\otimes N$ and the linear basis of $M_{|G|^2}(\Bbb{C})$ for convenience, respectively.
\end{remark}

Similar to the case of the field algebra ${\mathcal{F}}$ of $G$-spin models, one can show that $M_{|G|^2}({\mathcal{F}})$ is the $C^*$-algebra.
From now on, we call $M_{|G|^2}({\mathcal{F}})$ the field algebra of  $M_{|G|^2}(G)$-spin models, and we call $O_M^g(x)$ and $D_N^h(l)$ the order and disorder operators, respectively.

Now one can show that the field algebra ${\mathcal{F}}\rtimes D(G)\rtimes \widehat{D(G)}$ is $D(G)$-module algebra. Indeed,
the map $$\tau\colon D(G)\times ({\mathcal{F}}\rtimes D(G)\rtimes\widehat{D(G)}) \rightarrow {\mathcal{F}}\rtimes D(G)\rtimes\widehat{D(G)}$$
given on the generating elements of ${\mathcal{F}}\rtimes D(G)\rtimes\widehat{D(G)}$ as
$$\tau((g,h)\times(\widetilde{F}\otimes(y,\delta_x) ))=\delta_{h^{-1}gh,x^{-1}yx}(\widetilde{F}\otimes(y,\delta_{xh^{-1}}))$$
for any $\widetilde{F}\in {\mathcal{F}}\rtimes D(G)$, can be linearly extended both in $D(G)$ and ${\mathcal{F}}\rtimes D(G)\rtimes \widehat{D(G)}$.

The observable algebra of $M_{|G|^2}(G)$-spin models is defined as  $({\mathcal{F}}\rtimes D(G)\rtimes \widehat{D(G)})^{D(G)}$.
So it is clear that $({\mathcal{F}}\rtimes D(G)\rtimes \widehat{D(G)})^{D(G)}=
E_2({\mathcal{F}}\rtimes D(G)\rtimes \widehat{D(G)})={\mathcal{F}}\rtimes D(G)$.

\begin{remark}
Let ${\mathcal{A}}\subseteq{\mathcal{F}}$ be an inclusion of unital $C^*$-algebras with a conditional expectation $E\colon{\mathcal{F}}\rightarrow{\mathcal{A}}$ of index-finite type \cite{XINQL}. Set
${\mathcal{F}}_{-1}={\mathcal{A}}, \ {\mathcal{F}}_0={\mathcal{F}}$, and
$E_1=E$, and recall the $C^*$-basic construction (the $C^*$-algebra version of the basic construction). We inductively define $e_{k+1}=e_{_{\mathcal{F}_{k-1}}}$
and ${\mathcal{F}}_{k+1}=\langle{\mathcal{F}}_k,e_{k+1}\rangle_{C^*}$, the Jones projection and $C^*$-basic
construction applied to $E_{k+1}\colon {\mathcal{F}}_{k}\rightarrow{\mathcal{F}}_{k-1}$, and take
$E_{k+2}\colon {\mathcal{F}}_{k+1}\rightarrow{\mathcal{F}}_{k}$ to be the
dual conditional expectation $E_{{\mathcal{F}}_{k}}$ of Definition 2.3.3 in \cite{Y.Wat}. Then this gives the inclusion tower
of iterated basic constructions
$${\mathcal{A}}\subseteq{\mathcal{F}}\subseteq{\mathcal{F}}\rtimes D(G)\subseteq {\mathcal{F}}\rtimes D(G)\rtimes \widehat{D(G)}\subseteq
{\mathcal{F}}\rtimes D(G)\rtimes\widehat{D(G)}\rtimes D(G)\subseteq\cdots.$$
 It follows from Proposition 2.10.11 in \cite{Y.Wat} that this tower
does not depend on the choice of $E$.

Notice that ${\mathcal{F}}_2={\mathcal{F}}\rtimes D(G)\rtimes \widehat{D(G)}$ is $C^*$-isomorphic to $M_{|G|^2}({\mathcal{F}})$,
 the field algebra of a $M_{|G|^2}(G)$-spin model, and
${\mathcal{F}}_4={\mathcal{F}}\rtimes D(G)\rtimes \widehat{D(G)}\rtimes D(G)\rtimes \widehat{D(G)}$ is $C^*$-isomorphic to $M_{|G|^4}({\mathcal{F}})$, called the field algebra of a $M_{|G|^4}(G)$-spin model, where the order and disorder operators can be defined similar to those in Remark 4.2.
\end{remark}

 \end{document}